\documentclass[10pt]{article}
\usepackage{lmodern}
\usepackage{amsmath}
\usepackage[T1]{fontenc}
\usepackage[utf8]{inputenc}
\usepackage{authblk}
\usepackage{amsfonts}
\usepackage{graphicx}
\usepackage{rotating}
\usepackage{amssymb}
\usepackage[english]{babel}
\usepackage{xcolor}
\usepackage{amsthm}
\usepackage{graphicx}
\usepackage{mathrsfs}
\usepackage{makecell}
\usepackage{microtype}
\usepackage{mathscinet}
\usepackage{array}
\usepackage{multirow}
\usepackage{enumerate}
\usepackage[cal=boondoxo,bb=ams]{mathalfa}
\usepackage{hyperref}
\hypersetup{hidelinks}
\usepackage{booktabs}%table
\newtheorem{theorem}{Theorem}[section]
\newtheorem{prop}{Proposition}[section]

\newtheorem{remark}{Remark}[section]

\newcommand{\ml}{\mathcal}
\newcommand{\mb}{\mathbb}
\DeclareMathOperator{\intt}{int}
\DeclareMathOperator{\extt}{ext}
\DeclareMathOperator{\bdd}{bdd}
\DeclareMathOperator{\lin}{lin}
\DeclareMathOperator{\nlin}{nlin}

\def\XXint#1#2#3{{\setbox0=\hbox{$#1{#2#3}{\int}$ }
		\vcenter{\hbox{$#2#3$ }}\kern-.6\wd0}}

\title{Sharp behavior of semilinear damped wave equations driven by mixed local-nonlocal operators}
\author[1]{Wenhui Chen\thanks{Wenhui Chen (wenhui.chen.math@gmail.com)}}
\affil[1]{School of Mathematics and Information Science, Guangzhou University,\authorcr Guangzhou, P.R. China}
\author[2]{Tuan Anh Dao\thanks{Tuan Anh Dao (anh.daotuan@hust.edu.vn)}}
\affil[2]{Faculty of Mathematics and Informatics, Hanoi University of Science and Technology,\authorcr Hanoi, Vietnam}

\date{}

\begin{document}
		\maketitle

		\begin{abstract}
			\medskip
This paper investigates the Cauchy problem for the semilinear damped wave equation $u_{tt}+\mathcal{L}_{a,b}u+u_t=|u|^p$ with the mixed local-nonlocal operator $\mathcal{L}_{a,b}:=-a\Delta+b(-\Delta)^{\sigma}$, where $a,b\in\mathbb{R}_+$ and $\sigma\in(0,1)\cup (1,+\infty)$. We determine the critical exponent for this problem being $p_{\mathrm{crit}}=1+\frac{2\min\{1,\sigma\}}{n}$, which sharply separates global in-time existence and finite-time blow-up of solutions. Furthermore, for the super-critical case $p>p_{\mathrm{crit}}$, we establish the asymptotic profiles of global in-time solutions, showing the anomalous diffusion when $\sigma\in(0,1)$ and the classical diffusion when $\sigma\in(1,+\infty)$, together with the sharp decay estimates. For solutions blowing up in finite time when $1<p\leqslant p_{\mathrm{crit}}$, we derive the sharp estimates for upper and lower bounds of lifespan. Our results reveal the crucial influence of mixed operators on the qualitative properties of solutions, fundamentally governing their critical phenomena, large-time behavior and blow-up dynamics, via $\max\{1,\sigma\}$ or $\min\{1,\sigma\}$.
			\\
			
			
\noindent\textbf{AMS Classification (2020)} 35L71, 35B33, 35A01, 35B44, 35B40
		\end{abstract}
\fontsize{12}{15}
\selectfont

\section{Introduction}\setcounter{equation}{0}\label{Section-Introduction}
\hspace{5mm}Let $\ml{L}_{a,b}$ be the mixed local-nonlocal operator
\begin{align*}
	\ml{L}_{a,b}:=-a\Delta+b(-\Delta)^{\sigma}\ \ \text{with}\ \ a,b\in\mb{R}_+,
\end{align*}
where $(-\Delta)^{\sigma}$ denotes the fractional Laplacian of order $\sigma\in(0,1)\cup (1,+\infty)$, which always offers mathematical challenges caused by the lack of invariance under scaling. The operator $\ml{L}_{a,b}$ arises naturally in the scaling limit of stochastic processes \cite[Appendix]{Dipierro-Proietti-Valdinoci=2023} that combine the classical Brownian motion (modeled by $-a\Delta$) and the L\'evy flight (modeled by $b(-\Delta)^{\sigma}$). It provides a unified framework to analyze the interplay between local and nonlocal diffusion effects. This mixed operator has some applications in diverse fields, for example, the effectiveness of different pandemic containment strategies (comparing ``regional'' versus ``global'' restrictions) \cite{Dipierro-Proietti-Valdinoci=2023}, the  bi-model power law distribution processes \cite{Pagnini-Vitali=2021}, and the optimal animal foraging patterns \cite{Dipierro-Valdinoci=2021,Dipierro-Proietti-Valdinoci=2023}. The mathematical analysis of semilinear elliptic and parabolic equations involving the operator $\ml{L}_{a,b}$ is well-developed in  \cite{Biagi-Dipierro-Vecchi=2022,Dipierro-Proietti-Valdinoci=2023,Biagi-Dipierro-Valdinoci=2023,Biagi-Punzo-Vecchi=2025,Del-Pezzo-Ferreira=2025} and references given therein.
For instance, the critical exponent $p_{\mathrm{Heat}}:=1+\frac{2\sigma}{n}$ has been verified by \cite{Biagi-Punzo-Vecchi=2025} recently for the parabolic equation
\begin{align}\label{Eq-Parabolic}
\begin{cases}
	u_t+\ml{L}_{a,b}u=u^p,&x\in\mb{R}^n,\ t\in\mb{R}_+,\\
	u(0,x)= u_0(x),&x\in\mb{R}^n,
\end{cases}
\end{align}
carrying $\sigma\in(0,1)$ and $p>1$. However, to the best of our knowledge, the role of $\ml{L}_{a,b}$ for hyperbolic equations seems unexplored from the literature. It motivates us in this manuscript to study a typical example, that is a classical damped wave equation.

This paper is devoted to the study of semilinear damped wave equations driven by the mixed local-nonlocal operator $\ml{L}_{a,b}$ in the whole space $\mb{R}^n$, namely,
\begin{align}\label{Eq-Nonlinear-Local-Nonlocal}
\begin{cases}
u_{tt}+\ml{L}_{a,b}u+u_t=|u|^p,&x\in\mb{R}^n,\ t\in\mb{R}_+,\\
u(0,x)=\varepsilon u_0(x),\ u_t(0,x)=\varepsilon u_1(x),&x\in\mb{R}^n,
\end{cases}
\end{align}
with the power exponent $p>1$, where $\varepsilon$ is a positive parameter describing the magnitude of the initial data. Notice that \eqref{Eq-Parabolic} is a singular limit model of \eqref{Eq-Nonlinear-Local-Nonlocal} by taking the vanishing coefficient of $u_{tt}$. Let $\sigma_{\max}:=\max\{1,\sigma\}$ and $\sigma_{\min}:=\min\{1,\sigma\}$. Our main contributions for the semilinear Cauchy problem \eqref{Eq-Nonlinear-Local-Nonlocal} are three-fold as follows:
\begin{description}
	\item[Contribution 1.] We determine the critical exponent
	\begin{align}\label{critical-exponent}
		p_{\mathrm{crit}}=p_{\mathrm{crit}}(n,\sigma_{\min}):=1+\frac{2\sigma_{\min}}{n}.
	\end{align}
This critical exponent serves as a sharp threshold for the nonlinearity power $p$ such that for the sub-critical case $1<p<p_{\mathrm{crit}}$, all non-trivial weak solutions with the initial data of fixed sign in $L^1$ blow up in finite time, regardless of the size of initial data;  whereas for the super-critical case $p>p_{\mathrm{crit}}$, a unique global in-time Sobolev solution exists for sufficiently small initial data; and the critical case $p=p_{\mathrm{crit}}$ belongs to the blow-up range as well.
	\item[Contribution 2.] We establish the large-time asymptotic profile
	\begin{align*}
		u(t,\cdot)\sim G(t,\cdot)\int_{\mb{R}^n}\Big(\varepsilon u_0(x)+\varepsilon u_1(x)+\int_0^{+\infty}|u(t,x)|^p\,\mathrm{d}t\Big)\mathrm{d}x \ \ \text{when}\ \ p>p_{\mathrm{crit}}
	\end{align*}
	in the $\dot{H}^s$ norm with $s\in[0,\sigma_{\min}]$, where the dominant function, being the anomalous diffusion when $\sigma\in(0,1)$ or the classical diffusion when $\sigma\in(1,+\infty)$, is given by
	\begin{align*}
		G(t,x):=\begin{cases}
			\ml{F}_{\xi\to x}^{-1}\big(\mathrm{e}^{-b|\xi|^{2\sigma}t}\big)&\text{if}\ \  \sigma\in(0,1),\\[0.5em]
			\ml{F}_{\xi\to x}^{-1}\big(\mathrm{e}^{-a|\xi|^{2}t}\big)&\text{if}\ \  \sigma\in(1,+\infty).
		\end{cases}
	\end{align*}
	It can be understood by the Fourier multiplier $\ml{G}(t,|D|)$, namely, $\widehat{\ml{G}}(t,|\xi|)=\widehat{G}(t,\xi)$.
	As our byproduct, we prove the following sharp polynomial decay estimate:
	\begin{align*}
		\|u(t,\cdot)\|_{\dot{H}^s}\approx t^{-\frac{n+2s}{4\sigma_{\min}}}
	\end{align*}
	for large-time $t\gg1$, assuming suitable $L^1$ integrable initial data.
	\item[Contribution 3.] We derive the sharp lifespan estimates
	\begin{align}\label{Sharp-Lifespan}
		T_{\varepsilon} \begin{cases}
			=+\infty&\text{if}\ \ p>p_{\mathrm{crit}},\\
			\approx\Gamma(\varepsilon;p,\sigma_{\min})&\text{if}\ \ p\leqslant p_{\mathrm{crit}},
		\end{cases}
	\end{align}
	with the $\varepsilon$-dependent (decreasing) function
	\begin{align*}
	\Gamma(\varepsilon;p,\sigma_{\min}):=\begin{cases}
		\exp(C\varepsilon^{-(p-1)})&\text{if}\ \ p=p_{\mathrm{crit}},\\
		\varepsilon^{-\frac{2\sigma_{\min}(p-1)}{2\sigma_{\min}-n(p-1)}}&\text{if}\ \ p<p_{\mathrm{crit}},
	\end{cases}
	\end{align*}
	where $T_{\varepsilon}$ denotes the maximal time for which a unique local in-time solution exists. Here $C$ stands for a suitable positive constant independent of $\varepsilon$.
\end{description}

To contextualize our work, we first recall some known results for the damped wave equation in two well-studied limiting cases with respect to the parameters $a$ or $b$: the local case driven solely by the classical Laplacian $-\Delta$ and the general nonlocal case driven solely by the fractional Laplacian $(-\Delta)^\sigma$. Our focus will be on results for the $L^1$ integrable initial data.

For the local case $(a,b)=(1,0)$, our general model \eqref{Eq-Nonlinear-Local-Nonlocal} reduces to the following semilinear classical damped wave equation:
\begin{align}\label{Eq-Nonlinear-Damped-Wave}
	\begin{cases}
		u_{tt}-\Delta u+u_t=|u|^p,&x\in\mb{R}^n,\ t\in\mb{R}_+,\\
		u(0,x)=\varepsilon u_0(x),\ u_t(0,x)=\varepsilon u_1(x),&x\in\mb{R}^n.
	\end{cases}
\end{align}
Through Matsumura-type estimates, weighted energy methods and test function methods, the authors of \cite{Matsumura-1976,Todorova-Yordanov-2001,Zhang-2001,Ikehata-Tanizawa-2005} determined its critical exponent $p_{\mathrm{Fuj}}(n):=1+\frac{2}{n}$ (the so-called Fujita exponent arising in the semilinear diffusion equation \cite{Fujita-1966}).
Subsequently, the sharp lifespan estimates
\begin{align}\label{Life_span}
	T_{\varepsilon}\begin{cases}
		=+\infty &\text{if}\ \ p>p_{\mathrm{Fuj}}(n),\\
		\approx \exp(C\varepsilon^{-(p-1)})&\text{if}\ \ p=p_{\mathrm{Fuj}}(n),\\
		\approx \varepsilon^{-\frac{2(p-1)}{2-n(p-1)}}&\text{if} \ \ p<p_{\mathrm{Fuj}}(n),
	\end{cases}
\end{align}
were derived by  \cite{LS-Li-Zhou=1995,LS-Kirane-Qafsaoui=2002,LS-Nishihara=2011,LS-Ikeda-Wakasugi=2015,LS-Ikeda-Ogawa=2016,LS-Fujiwara-Ikeda-Wakasugi=2019,LS-Ikeda-Sobajima=2019,LS-Lai-Zhou=2019} via comparison arguments, modified test function methods and spatial weighted $H^s$ estimates.

For the nonlocal case $(a,b)=(0,1)$, our general model \eqref{Eq-Nonlinear-Local-Nonlocal} reduces to the following semilinear frictional damped $\sigma$-evolution equation (only the situation $\sigma\in[1,+\infty)$ was considered in the literature):
\begin{align}\label{Eq-Nonlinear-Damped-Sigma}
	\begin{cases}
		u_{tt}+(-\Delta)^{\sigma} u+u_t=|u|^p,&x\in\mb{R}^n,\ t\in\mb{R}_+,\\
		u(0,x)=\varepsilon u_0(x),\ u_t(0,x)=\varepsilon u_1(x),&x\in\mb{R}^n,
	\end{cases}
\end{align}
whose critical exponent is determined by the so-called modified Fujita exponent $p_{\mathrm{Fuj}}(\frac{n}{\sigma}):=1+\frac{2\sigma}{n}$ by using $L^m-L^q$ estimates (with suitable values $1\leqslant m\leqslant q\leqslant +\infty$) and modified test function methods \cite{D'Abbicco-Ebert=2014, Pham-K-Reissig=2015, D'Abbicco-Ebert=2017, D'Abbicco-Fujiwara=2021, Dao-Reissig=2021-01}. Recently, the sharp lifespan estimates
\begin{align}\label{Life_span-2}
	T_{\varepsilon}\begin{cases}
		=+\infty &\text{if}\ \ p>p_{\mathrm{Fuj}}(\frac{n}{\sigma}),\\
		\approx \exp(C\varepsilon^{-(p-1)})&\text{if}\ \ p=p_{\mathrm{Fuj}}(\frac{n}{\sigma}),\\
		\approx \varepsilon^{-\frac{2\sigma(p-1)}{2\sigma-n(p-1)}}&\text{if} \ \ p<p_{\mathrm{Fuj}}(\frac{n}{\sigma}),
	\end{cases}
\end{align}
were completely demonstrated by introducing a new test function in \cite{Chen-Girardi=2025}. Notably, in the limit case $\sigma = 1$, these results  exactly coincide with those for the classical damped wave equation \eqref{Eq-Nonlinear-Damped-Wave}.

This paper is devoted to the analysis of the mixed local-nonlocal problem \eqref{Eq-Nonlinear-Local-Nonlocal} that interpolates between the purely local model \eqref{Eq-Nonlinear-Damped-Wave} and the purely nonlocal model \eqref{Eq-Nonlinear-Damped-Sigma}. Then, a fundamental and natural question arises.
\begin{center}
\textit{How do the competing effects of a local operator $-a\Delta$ versus a nonlocal operator $b(-\Delta)^\sigma$\\ determine qualitative properties of solutions?}
\end{center}
We seek to solve this question by providing a complete description of solutions' sharp behavior, encompassing global in-time properties (existence, regularities, decay rates, asymptotic profiles) and finite-time blow-up phenomena (critical exponents, lifespan estimates).
\medskip

\noindent\textbf{\large Notation. }Let the generic positive constants $c$ and $C$, which are independent of $t$ and $\varepsilon$, vary from line to line. The notation $f\lesssim g$ means $f\leqslant Cg$ for  $C>0$; and $f\gtrsim g$ is defined similarly. We write $f\approx g$ to denote a sharp estimate, meaning that $g\lesssim f\lesssim g$ holds. We introduce the regions
\begin{align*}
	\ml{Z}_{\intt}(\varepsilon_0):=\{|\xi|\leqslant\varepsilon_0\ll1\},\  \
	\ml{Z}_{\bdd}(\varepsilon_0,N_0):=\{\varepsilon_0\leqslant |\xi|\leqslant N_0\},\ \  
	\ml{Z}_{\extt}(N_0):=\{|\xi|\geqslant N_0\gg1\}.
\end{align*}
Moreover, we define the smooth cut-off functions $\chi_{\intt}(\xi),\chi_{\bdd}(\xi),\chi_{\extt}(\xi)$ having their supports in the corresponding regions $\ml{Z}_{\intt}(\varepsilon_0)$, $\ml{Z}_{\bdd}(\varepsilon_0/2,2N_0)$ and $\ml{Z}_{\extt}(N_0)$, respectively, which form the partition of unity $\chi_{\bdd}(\xi)=1-\chi_{\intt}(\xi)-\chi_{\extt}(\xi)$. For $s \in \mathbb{R}$,  $|D|^s$ is the Fourier multiplier operator   defined by $\mathcal{F}(|D|^s f) = |\xi|^s \mathcal{F}(f)$ for $f\in\ml{S}$, where $\mathcal{F}$ denotes the Fourier transform. The fractional Laplacian is defined via $(-\Delta)^{s}:=|D|^{2s}$ for $s\in\mb{R}$. For integrable functions $f$, we define the averages
\begin{align*}
P_f:=\int_{\mb{R}^n}f(x)\,\mathrm{d}x\ \ \text{and}\ \  \ml{P}_f:=\int_0^{+\infty}\int_{\mb{R}^n}f(t,x)\,\mathrm{d}x\,\mathrm{d}t. 
\end{align*}
For the sake of convenience, we introduce the combined initial data
\begin{align*}
u_{0+1}(x):=u_0(x)+u_1(x).
\end{align*}
We denote by $p':=\frac{p}{p-1}$, the H\"older conjugate of $p$ and by $\langle x\rangle:=\sqrt{1+|x|^2}$, the Japanese bracket.

\section{Main results}\setcounter{equation}{0}\label{Section-Main-Results}
\subsection{Critical exponent and large-time behavior}
\hspace{5mm}We now present our first main result concerning the global in-time existence of small data Sobolev solutions in the super-critical regime $p>p_{\mathrm{crit}}$.
\begin{theorem}\label{Thm-GESDS}
Let $(u_0,u_1)\in (H^{\sigma_{\min}}\cap L^1)\times (L^2\cap L^1)$. Assume $p\geqslant 2$, $p\leqslant\frac{n}{n-2\sigma_{\min}}$ if $n>2\sigma_{\min}$, and $p>p_{\mathrm{crit}}$.
Then, there exists $\varepsilon_1>0$ such that for any $\varepsilon\in(0,\varepsilon_1]$ the semilinear Cauchy problem \eqref{Eq-Nonlinear-Local-Nonlocal} admits a uniquely global in-time solution $u\in\ml{C}([0,+\infty),H^{\sigma_{\min}})$. Furthermore, the solution satisfies the following decay estimate for $s\in[0,\sigma_{\min}]$:
\begin{align}\label{Est-Decay}
\|u(t,\cdot)\|_{\dot{H}^{s}}&\lesssim \varepsilon (1+t)^{-\frac{n+2s}{4\sigma_{\min}}}\|(u_0,u_1)\|_{(H^{\sigma_{\min}}\cap L^1)\times (L^2\cap L^1)}.
\end{align}
\end{theorem}
\begin{remark}\label{Rem-First-Competition}
	The regularity of global in-time solutions can be improved by assuming $u_0\in H^{\sigma_{\max}}$ additionally. Let us first replace the restriction ``$p\leqslant\frac{n}{n-2\sigma_{\min}}$ if $n>2\sigma_{\min}$'' by the weaker one ``$p\leqslant\frac{n}{n-2\sigma_{\max}}$ if $n>2\sigma_{\max}$'' in Theorem \ref{Thm-GESDS}. If $\sigma\in[\frac{1}{2},2]$, then  $u\in\ml{C}([0,+\infty),H^{\sigma_{\max}})$ and the decay estimates \eqref{Est-Decay} hold for $s\in[0,\sigma_{\max}]$. For the remaining case $\sigma\in(0,\frac{1}{2})\cup(2,+\infty)$, achieving this higher regularity requires the stronger condition $p\geqslant 1+\frac{\sigma_{\max}}{n}$, which is more restrictive than the super-critical condition $p>p_{\mathrm{crit}}$ in Theorem \ref{Thm-GESDS}. This dichotomy reveals the first competing influence of local-nonlocal operators as follows: the maximal spatial regularity of solutions is ultimately determined by the component of $\ml{L}_{a,b}$ with the larger exponent $\sigma_{\max}$.
\end{remark}

The second main result precisely describes the asymptotic profiles of global in-time Sobolev solutions and rigorously justifies the sharpness of decay rates in \eqref{Est-Decay} for large-time.
\begin{theorem}\label{Thm-Large-Time}
Assume $u_0$, $u_1$ and  $p$ satisfy the hypotheses of Theorem \ref{Thm-GESDS}.
 Then, the obtained global in-time small data Sobolev solution $u$ satisfies the following approximation:
\begin{align*}
\lim\limits_{t\to+\infty} t^{\frac{n+2s}{4\sigma_{\min}}}\big\|u(t,\cdot)-G(t,\cdot)\big(\varepsilon P_{u_{0+1}}+\ml{P}_{|u|^p}\big)\big\|_{\dot{H}^s}=0
\end{align*}
for $s\in[0,\sigma_{\min}]$. Moreover, the sharp large-time decay estimates
\begin{align*}
t^{-\frac{n+2s}{4\sigma_{\min}}}\big|\varepsilon P_{u_{0+1}}+\ml{P}_{|u|^p}\big|\lesssim\|u(t,\cdot)\|_{\dot{H}^s}\lesssim \varepsilon t^{-\frac{n+2s}{4\sigma_{\min}}}\|(u_0,u_1)\|_{(H^{\sigma_{\min}}\cap L^1)\times (L^2\cap L^1)}
\end{align*}
hold for $s\in[0,\sigma_{\min}]$, provided that $\varepsilon P_{u_{0+1}}+\ml{P}_{|u|^p}\neq0$.
\end{theorem}
\begin{remark}\label{Rem-Second-Competition}
These dichotomies in $G(t,x)$ and $t^{-\frac{n+2s}{4\sigma_{\min}}}$ reveal the second competing influence of local-nonlocal operators as follows: the large-time asymptotic behavior (profile and sharp decay rate) of solutions are governed by the component of $\ml{L}_{a,b}$ with the smaller exponent $\sigma_{\min}$.
\end{remark}

Let us turn our consideration to the blow-up phenomena in the remaining regime $p\leqslant p_{\mathrm{crit}}$.
\begin{theorem}\label{Thm-Blow-up}
Let $u_{0+1}\in L^1$ such that $\int_{\mb{R}^n}u_{0+1}(x)\,\mathrm{d}x>0$.
Assume $p\leqslant p_{\mathrm{crit}}$.
Then, there is no global in-time solution $u \in \mathcal{C}([0,+\infty),L^2)$ to the semilinear Cauchy problem \eqref{Eq-Nonlinear-Local-Nonlocal}. Furthermore, its lifespan $T_{\varepsilon}$ is bounded from above by
	\begin{align}\label{Est-upper}
		T_{\varepsilon}\lesssim \Gamma(\varepsilon;p,\sigma_{\min}).
	\end{align}
\end{theorem}

\begin{remark}\label{Rem-Third-Competition}
The results of Theorem \ref{Thm-GESDS} (global in-time existence for $p > p_{\mathrm{crit}}$) and Theorem \ref{Thm-Blow-up} (finite-time blow-up for $p \leqslant p_{\mathrm{crit}}$) collectively demonstrate the critical exponent \eqref{critical-exponent} for the $L^1$ integrable initial data, which can be equivalently expressed as the Fujita-type exponent in the fractional spatial dimension $\frac{n}{\sigma_{\min}}$, namely, $p_{\mathrm{crit}}=p_{\mathrm{Fuj}}(\frac{n}{\sigma_{\min}})$. This coincides with the critical exponents for \eqref{Eq-Nonlinear-Damped-Sigma} if $\sigma\in(0,1)$, and for \eqref{Eq-Nonlinear-Damped-Wave} if $\sigma\in(1,+\infty)$. This dichotomy reveals the third competing influence of local-nonlocal operators as follows: the critical exponent is determined solely by the component of $\ml{L}_{a,b}$ with the smaller exponent $\sigma_{\min}$.
\end{remark}

\subsection{Sharp lifespan estimates}
\hspace{5mm}As established in Theorem \ref{Thm-Blow-up}, every non-trivial local in-time solution to \eqref{Eq-Nonlinear-Local-Nonlocal} with $1<p\leqslant p_{\mathrm{crit}}$ will blow up in finite time. Moreover, the sharp upper bound estimates of $T_{\varepsilon}$ have been shown. We now complement this result by deriving lower bound estimates for the lifespan $T_{\varepsilon}$.
\begin{theorem}\label{Thm-Lower-Lifespan}
Let $(u_0,u_1)\in (H^{\sigma_{\min}}\cap L^1)\times (L^2\cap L^1)$. Assume $p\geqslant 2$, $p\leqslant\frac{n}{n-2\sigma_{\min}}$ if $n>2\sigma_{\min}$. Then, the semilinear Cauchy problem \eqref{Eq-Nonlinear-Local-Nonlocal} admits a uniquely local in-time solution $u\in\ml{C}([0,T],H^{\sigma_{\min}})$ for any $\varepsilon>0$, satisfying the estimate \eqref{Est-Decay} for any $t\in[0,T]$, where $T\lesssim \Gamma(\varepsilon;p,\sigma_{\min})$.
That is to say, its lifespan $T_{\varepsilon}$ is bounded from below by
\begin{align}\label{Est-lower}
	T_{\varepsilon}\gtrsim\Gamma(\varepsilon;p,\sigma_{\min}).
\end{align}
\end{theorem}
\begin{remark}\label{Rem-Fourth-Competition}
Combining the upper bound \eqref{Est-upper} with the lower bound \eqref{Est-lower} above, we conclude the sharp lifespan estimates \eqref{Sharp-Lifespan}. These coincide with the sharp lifespan estimates \eqref{Life_span-2} if $\sigma\in(0,1)$, and \eqref{Life_span} if $\sigma\in(1,+\infty)$. This dichotomy reveals the fourth competing influence of local-nonlocal operators as follows: the sharp lifespan estimates are governed by the component of $\ml{L}_{a,b}$ with the smaller exponent $\sigma_{\min}$.
\end{remark}

\section{Linearized damped wave equations    driven by $\mathcal{L}_{a,b}$}\setcounter{equation}{0}\label{Section-Linearized}
\hspace{5mm}As a crucial preparation for analyzing global in-time behavior (including existence, sharp decay estimate and asymptotic profile) of solutions to the semilinear model \eqref{Eq-Nonlinear-Local-Nonlocal}, we first study its corresponding linearized Cauchy problem
\begin{align}\label{Eq-Linear-Local-Nonlocal}
\begin{cases}
v_{tt}+\ml{L}_{a,b}v+v_t=0,&x\in\mb{R}^n,\ t\in\mb{R}_+,\\
v(0,x)= v_0(x),\ v_t(0,x)= v_1(x),&x\in\mb{R}^n.
\end{cases}
\end{align}
Our main tools are the WKB analysis as well as the Fourier analysis.

\subsection{Refined estimates of solutions in the Fourier space}
\hspace{5mm}Thanks to $\ml{F}_{x\to\xi}(\ml{L}_{a,b}v)=(a|\xi|^2+b|\xi|^{2\sigma})\widehat{v}$, we can apply the partial Fourier transform with respect to $x\in\mb{R}^n$ to the linearized Cauchy problem \eqref{Eq-Linear-Local-Nonlocal} as usual. Then, its characteristic equation $\lambda^2+\lambda+(a|\xi|^2+b|\xi|^{2\sigma})=0$ has the following two roots:
\begin{align*}
\lambda_{\pm}=\frac{1}{2}\big(-1\pm\sqrt{1-4(a|\xi|^2+b|\xi|^{2\sigma})}\,\big). 
\end{align*}
\underline{Small Frequencies:} For $\xi\in\ml{Z}_{\intt}(\varepsilon_0)$, the characteristic roots are expanded by 
	\begin{align*}
	\lambda_+=\begin{cases}
		-b|\xi|^{2\sigma}+O(|\xi|^{\min\{2,4\sigma\}})&\text{if}\ \ \sigma\in(0,1),\\
		-a|\xi|^2+O(|\xi|^{\min\{4,2\sigma\}})&\text{if}\ \ \sigma\in(1,+\infty),
	\end{cases}
	\quad \lambda_-=\begin{cases}
		-1+O(|\xi|^{2\sigma})&\text{if}\ \ \sigma\in(0,1),\\
		-1+O(|\xi|^2)&\text{if}\ \ \sigma\in(1,+\infty).
	\end{cases}
\end{align*}
Due to the pairwise distinct roots of quadratic equation, we have the representation formula
\begin{align}\label{Rep-Fourier}
	\widehat{v}=\widehat{\ml{K}}_0(t,|\xi|)\widehat{v}_0+\widehat{\ml{K}}_1(t,|\xi|)\widehat{v}_1:=\frac{\lambda_+\mathrm{e}^{\lambda_-t}-\lambda_-\mathrm{e}^{\lambda_+t}}{\lambda_+-\lambda_-}\widehat{v}_0+\frac{\mathrm{e}^{\lambda_+t}-\mathrm{e}^{\lambda_-t}}{\lambda_+-\lambda_-}\widehat{v}_1.
\end{align}
\begin{prop}\label{Prop-Small}
Let $\xi\in\ml{Z}_{\intt}(\varepsilon_0)$. If $\sigma\in(0,1)$, then the solution satisfies
\begin{align*}
	\chi_{\intt}(\xi)|\widehat{v}|&\lesssim \chi_{\intt}(\xi)\mathrm{e}^{-c|\xi|^{2\sigma}t}(|\widehat{v}_0|+|\widehat{v}_1|),\\
	\chi_{\intt}(\xi)\big|\widehat{v}-\mathrm{e}^{-b|\xi|^{2\sigma}t}(\widehat{v}_0+\widehat{v}_1)\big|&\lesssim \chi_{\intt}(\xi)\big(\mathrm{e}^{-ct}+|\xi|^{\min\{2-2\sigma,2\sigma\}}\mathrm{e}^{-c|\xi|^{2\sigma}t}\big)(|\widehat{v}_0|+|\widehat{v}_1|).
\end{align*}
If $\sigma\in(1,+\infty)$, then the solution satisfies
\begin{align*}
	\chi_{\intt}(\xi)|\widehat{v}|&\lesssim \chi_{\intt}(\xi)\mathrm{e}^{-c|\xi|^{2}t}(|\widehat{v}_0|+|\widehat{v}_1|),\\
	\chi_{\intt}(\xi)\big|\widehat{v}-\mathrm{e}^{-a|\xi|^{2}t}(\widehat{v}_0+\widehat{v}_1)\big|&\lesssim \chi_{\intt}(\xi)\big(\mathrm{e}^{-ct}+|\xi|^{\min\{2,2\sigma-2\}}\mathrm{e}^{-c|\xi|^{2}t}\big)(|\widehat{v}_0|+|\widehat{v}_1|).
\end{align*}
\end{prop}
\begin{proof}
According to the representation of $\widehat{v}$ and the asymptotic expansions of $\lambda_{\pm}$, we notice the dominant part of solution, which can be rigorously proved by
\begin{align*}
	\chi_{\intt}(\xi)\Big|\widehat{v}-\frac{\mathrm{e}^{\lambda_+t}}{\lambda_+-\lambda_-}(-\lambda_-\widehat{v}_0+\widehat{v}_1)\Big|\lesssim\chi_{\intt}(\xi)(|\xi|^{2\sigma_{\min}}\mathrm{e}^{-ct}|\widehat{v}_0|+\mathrm{e}^{-ct}|\widehat{v}_1|),
\end{align*}
moreover,
\begin{align*}
	&\chi_{\intt}(\xi)\Big|\frac{\mathrm{e}^{\lambda_+t}}{\lambda_+-\lambda_-}(-\lambda_-\widehat{v}_0+\widehat{v}_1)-\mathrm{e}^{-b|\xi|^{2\sigma}t}(\widehat{v}_0+\widehat{v}_1)\Big|\\
	&\quad\lesssim \chi_{\intt}(\xi)\mathrm{e}^{-b|\xi|^{2\sigma}t}\left(\Big|\frac{\lambda_-(\mathrm{e}^{ O(|\xi|^{\min\{2,4\sigma\}})t}-1)+\lambda_+}{\lambda_+-\lambda_-}\Big||\widehat{v}_0|+\Big|\frac{\mathrm{e}^{ O(|\xi|^{\min\{2,4\sigma\}})t}-\lambda_++\lambda_-}{\lambda_+-\lambda_-}\Big||\widehat{v}_1|\right)\\
	&\quad\lesssim \chi_{\intt}(\xi)(|\xi|^{\min\{2,4\sigma\}}t+|\xi|^{2\sigma})\mathrm{e}^{-c|\xi|^{2\sigma}t}(|\widehat{v}_0|+|\widehat{v}_1|)\\
	&\quad\lesssim \chi_{\intt}(\xi)|\xi|^{\min\{2-2\sigma,2\sigma\}}\mathrm{e}^{-c|\xi|^{2\sigma}t}(|\widehat{v}_0|+|\widehat{v}_1|)
\end{align*}
if $\sigma\in(0,1)$; similarly,
\begin{align*}
	\chi_{\intt}(\xi)\Big|\frac{\mathrm{e}^{\lambda_+t}}{\lambda_+-\lambda_-}(-\lambda_-\widehat{v}_0+\widehat{v}_1)-\mathrm{e}^{-a|\xi|^{2}t}(\widehat{v}_0+\widehat{v}_1)\Big|\lesssim \chi_{\intt}(\xi)|\xi|^{\min\{2,2\sigma-2\}}\mathrm{e}^{-c|\xi|^{2}t}(|\widehat{v}_0|+|\widehat{v}_1|)
\end{align*}
if $\sigma\in(1,+\infty)$. Then, combining the last derived estimates and the triangle inequality, our desired pointwise estimates are established.
\end{proof}

\noindent\underline{Large Frequencies:} For $\xi\in\ml{Z}_{\extt}(N_0)$, the pairwise distinct characteristic roots are expanded by
	\begin{align*}
	\lambda_{\pm}=\begin{cases}
		\pm \sqrt{a} i|\xi|-\frac{1}{2}+iO(|\xi|^{2\sigma-1})&\text{if}\ \ \sigma\in(0,1),\\
		\pm \sqrt{b} i|\xi|^{\sigma}-\frac{1}{2}+iO(|\xi|^{2-\sigma})&\text{if}\ \ \sigma\in(1,+\infty).
	\end{cases}
\end{align*}
Thanks to the representation formula \eqref{Rep-Fourier}, we are able to get the next preliminary.
\begin{prop}\label{Prop-Large}
Let $\xi\in\ml{Z}_{\extt}(N_0)$. If $\sigma\in(0,1)$, then the solution satisfies
\begin{align*}
	\chi_{\extt}(\xi)|\widehat{v}|\lesssim\chi_{\extt}(\xi)\mathrm{e}^{-ct}(|\widehat{v}_0|+|\xi|^{-1}|\widehat{v}_1|).
\end{align*}
If $\sigma\in(1,+\infty)$, then the solution satisfies
\begin{align*}
\chi_{\extt}(\xi)|\widehat{v}|\lesssim\chi_{\extt}(\xi)\mathrm{e}^{-ct}(|\widehat{v}_0|+|\xi|^{-\sigma}|\widehat{v}_1|).
\end{align*}
\end{prop}

\medskip 
\noindent\underline{Bounded Frequencies.} For $\xi\in\ml{Z}_{\bdd}(\varepsilon_0,N_0)$, the characteristic roots fulfill $\text{Re}\,\lambda_{\pm}<0$ leading to
\begin{align*}
	\chi_{\bdd}(\xi)|\widehat{v}|\lesssim\chi_{\bdd}(\xi)\mathrm{e}^{-ct}(|\widehat{v}_0|+|\widehat{v}_1|).
\end{align*}

\subsection{Decay properties and asymptotic profiles}
\hspace{5mm}Let us recall $v_{0+1}(x):=v_0(x)+v_1(x)$. We next derive some $\dot{H}^s$ estimates of the solution.
\begin{prop}\label{Prop-Decay-Estimates}
Let $s\in[0,\sigma_{\max}]$. The solution of the Cauchy problem \eqref{Eq-Linear-Local-Nonlocal} satisfies the following upper bound estimates:
\begin{align*}
	\|v(t,\cdot)\|_{\dot{H}^s}&\lesssim(1+t)^{-\frac{n+2s}{4\sigma_{\min}}}\|(v_0,v_1)\|_{(H^s\cap L^1)\times (L^2\cap L^1)},\\
	\|v(t,\cdot)\|_{\dot{H}^s}&\lesssim(1+t)^{-\frac{s}{2\sigma_{\min}}}\|(v_0,v_1)\|_{H^s\times L^2}.
\end{align*}
Furthermore, it satisfies the refined estimate
\begin{align}\label{Eq-03}
\lim\limits_{t\to+\infty}t^{\frac{n+2s}{4\sigma_{\min}}}\|v(t,\cdot)-G(t,\cdot)P_{v_{0+1}}\|_{\dot{H}^s}=0,
\end{align}
and the lower bound estimate
\begin{align*}
\|v(t,\cdot)\|_{\dot{H}^s}\gtrsim t^{-\frac{n+2s}{4\sigma_{\min}}}|P_{v_{0+1}}|
\end{align*}
for large-time $t\gg1$, provided that $P_{v_{0+1}}\neq 0$.
\end{prop}
\begin{proof}
By using the polar coordinate, the equivalent relation
\begin{align}\label{Eq-04}
	\big\|\chi_{\intt}(\xi)|\xi|^s\mathrm{e}^{-c|\xi|^{2\theta}t}\big\|_{L^2}\approx\Big(\int_0^{\varepsilon_0}r^{2s+n-1}\mathrm{e}^{-2cr^{2\theta}t}\,\mathrm{d}r\Big)^{1/2}\approx t^{-\frac{n+2s}{4\theta}}
\end{align}
holds for $t\gg1$, which is bounded for $t\leqslant 1$.
 According to Propositions \ref{Prop-Small} and \ref{Prop-Large}, one derives
\begin{align*}
\|v(t,\cdot)\|_{\dot{H}^s}&\lesssim \big\|\chi_{\intt}(\xi)|\xi|^s\mathrm{e}^{-c|\xi|^{2\sigma_{\min}}t}\big\|_{L^2}\|(v_0,v_1)\|_{L^1\times L^1}+\mathrm{e}^{-ct}\|(v_0,v_1)\|_{\dot{H}^s\times \dot{H}^{\max\{s-\sigma_{\max},0\}}}\\
&\lesssim (1+t)^{-\frac{n+2s}{4\sigma_{\min}}}\|(v_0,v_1)\|_{L^1\times L^1}+\mathrm{e}^{-ct}\|(v_0,v_1)\|_{\dot{H}^s\times L^2}.
\end{align*}
Moreover, with the aid of $\displaystyle{\sup_{|\xi|\leqslant\varepsilon_0}|\xi|^s\mathrm{e}^{-c|\xi|^{2\theta}t}\lesssim (1+t)^{-\frac{s}{2\theta}}}$, we arrive at
\begin{align*}
\|v(t,\cdot)\|_{\dot{H}^s}&\lesssim \big\|\chi_{\intt}(\xi)|\xi|^s\mathrm{e}^{-c|\xi|^{2\sigma_{\min}}t}\big\|_{L^{\infty}}\|(v_0,v_1)\|_{L^2\times L^2}+\mathrm{e}^{-ct}\|(v_0,v_1)\|_{\dot{H}^s\times \dot{H}^{\max\{s-\sigma_{\max},0\}}}\\
&\lesssim (1+t)^{-\frac{s}{2\sigma_{\min}}}\|(v_0,v_1)\|_{L^2\times L^2}+\mathrm{e}^{-ct}\|(v_0,v_1)\|_{\dot{H}^s\times L^2}.
\end{align*}

We now turn to the large-time behavior as $t\gg1$. Similarly to the above calculations, we obtain
\begin{align*}
\|v(t,\cdot)-\ml{G}(t,|D|)v_{0+1}(\cdot)\|_{\dot{H}^s}&\lesssim\big\|\chi_{\intt}(\xi)|\xi|^s\big(\mathrm{e}^{-ct}+|\xi|^{\min\{2-2\sigma,2\sigma\}}\mathrm{e}^{-c|\xi|^{2\sigma}t}\big)\big\|_{L^2}\|(v_0,v_1)\|_{L^1\times L^1}\\
&\quad+\mathrm{e}^{-ct}\|(v_0,v_1)\|_{\dot{H}^s\times L^2}+\big\|\chi_{\extt}(\xi)|\xi|^s\mathrm{e}^{-b|\xi|^{2\sigma}t}\big\|_{L^{\infty}}\|(v_0,v_1)\|_{L^2\times L^2}\\
&\lesssim t^{-\frac{n+2s+2\min\{2-2\sigma,2\sigma\}}{4\sigma}}\|(v_0,v_1)\|_{L^1\times L^1}+\mathrm{e}^{-ct}\|(v_0,v_1)\|_{H^s\times L^2}
\end{align*}
if $\sigma\in(0,1)$; similarly,
\begin{align*}
\|v(t,\cdot)-\ml{G}(t,|D|)v_{0+1}(\cdot)\|_{\dot{H}^{s}}
&\lesssim t^{-\frac{n+2s+2\min\{2,2\sigma-2\}}{4}}\|(v_0,v_1)\|_{L^1\times L^1}+\mathrm{e}^{-ct}\|(v_0,v_1)\|_{H^s\times L^2}
\end{align*}
if $\sigma\in(1,+\infty)$. Therefore,
\begin{align}\label{Eq-01}
\|v(t,\cdot)-\ml{G}(t,|D|)v_{0+1}(\cdot)\|_{\dot{H}^s}=o\big(t^{-\frac{n+2s}{4\sigma_{\min}}}\big),
\end{align}
due to $\min\{2-2\sigma,2\sigma\}>0$ if $\sigma\in(0,1)$ and $\min\{2,2\sigma-2\}>0$ if $\sigma\in(1,+\infty)$.
 The mean value theorem shows
\begin{align*}
|G(t,x-y)-G(t,x)|\lesssim |y|\,|\nabla G(t,x-\theta_0y)|\ \  \text{with}\ \ \theta_0\in(0,1).
\end{align*}
From the separating line $h(t):=t^{\frac{1}{4\sigma_{\min}}}$, we employ the last two inequalities to get
\begin{align}\label{Eq-02}
&\|\ml{G}(t,|D|)v_{0+1}(\cdot)-G(t,\cdot)P_{v_{0+1}}\|_{\dot{H}^{s}}\notag\\
&\quad\lesssim\Big\|\int_{|y|\leqslant h(t)}[G(t,\cdot-y)-G(t,\cdot)]v_{0+1}(y)\,\mathrm{d}y\Big\|_{\dot{H}^s}+\Big\|\int_{|y|\geqslant h(t)}[|G(t,\cdot-y)|+|G(t,\cdot)|]\,|v_{0+1}(y)|\,\mathrm{d}y\Big\|_{\dot{H}^s}\notag\\
&\quad\lesssim h(t)\|\,|\xi|\widehat{\ml{G}}(t,|\xi|)\|_{\dot{H}^s}\|v_{0+1}\|_{L^1}+\|\widehat{\ml{G}}(t,|\xi|)\|_{\dot{H}^s}\|v_{0+1}\|_{L^1(|x|\geqslant h(t))}\notag\\
&\quad\lesssim t^{-\frac{n+2s+2}{4\sigma_{\min}}+\frac{1}{4\sigma_{\min}}}\|v_{0+1}\|_{L^1}+o\big(t^{-\frac{n+2s}{4\sigma_{\min}}}\big)=o\big(t^{-\frac{n+2s}{4\sigma_{\min}}}\big).
\end{align}
The combination of \eqref{Eq-01} and \eqref{Eq-02} is to verify \eqref{Eq-03}. Finally, let us consider the lower bound
\begin{align*}
\|v(t,\cdot)\|_{\dot{H}^s}&\geqslant \|G(t,\cdot)\|_{\dot{H}^{s}}|P_{v_{0+1}}|-\|v(t,\cdot)-G(t,\cdot)P_{v_{0+1}}\|_{\dot{H}^s}\gtrsim t^{-\frac{n+2s}{4\sigma_{\min}}}|P_{v_{0+1}}|
\end{align*}
for large-time $t\gg1$, provided that $P_{v_{0+1}}\neq 0$, where we used \eqref{Eq-04} and \eqref{Eq-03}.
\end{proof}

\section{Existence for semilinear damped wave equations driven by $\ml{L}_{a,b}$}\setcounter{equation}{0}\label{Section-Global}
\subsection{Philosophy of our proofs}
\hspace{5mm}Before proving the local/global in-time existence results and the sharp lower bound estimates of lifespan $T_{\varepsilon}$, as preparations, we at first explain our philosophy.

For any $T>0$, we introduce the evolution space of solutions by $X_T:=\ml{C}([0,T],H^{\sigma_{\min}})$ equipping the following time-weighted norm:
\begin{align*}
	\|u\|_{X_T}:=\sup\limits_{t\in[0,T]}\big((1+t)^{\frac{n}{4\sigma_{\min}}}\|u(t,\cdot)\|_{L^2}+(1+t)^{\frac{n+2\sigma_{\min}}{4\sigma_{\min}}}\|u(t,\cdot)\|_{\dot{H}^{\sigma_{\min}}}\big).
\end{align*}
According to Duhamel's principle, concerning the semilinear local-nonlocal equation \eqref{Eq-Nonlinear-Local-Nonlocal}, let us construct the nonlinear integral operator
\begin{align*}
	\ml{N}:\ u\in X_T\to\ml{N}[u]:=\varepsilon u^{\lin}+u^{\nlin},
\end{align*}
where $u^{\lin}$ denotes the solution of linearized Cauchy problem \eqref{Eq-Linear-Local-Nonlocal}, and $u^{\nlin}$ is defined by
\begin{align*}
	u^{\nlin}(t,x):=\int_0^t\ml{K}_1(t-\tau,|D|)|u(\tau,x)|^p\,\mathrm{d}\tau.
\end{align*}

Our assumption $(u_0,u_1)\in (H^{\sigma_{\min}}\cap L^1)\times (L^2\cap L^1)$ indicates $u^{\lin}\in X_T$ for any $T>0$ and the uniform estimate (i.e. Proposition \ref{Prop-Decay-Estimates} with $s\in\{0,\sigma_{\min}\}$)
\begin{align*}
	\|u^{\lin}\|_{X_T}\leqslant C_0\|(u_0,u_1)\|_{(H^{\sigma_{\min}}\cap L^1)\times (L^2\cap L^1)}
\end{align*}
with a suitable constant $C_0>0$. Our goal is to prove the existence results for a unique fixed point $u$ of the nonlinear integral operator $\ml{N}$ in the space $X_T$. Thus, we apply the Banach contraction principle for any $u$ and $\tilde{u}$ in the set
\begin{align*}
	\ml{B}_{\kappa}(X_T):=\big\{u\in X_T:\ \|u\|_{X_T}\leqslant\kappa:=2\varepsilon C_0\|(u_0,u_1)\|_{(H^{\sigma_{\min}}\cap L^1)\times (L^2\cap L^1)}\big\}.
\end{align*}
In other words, we will prove that the following fundamental inequalities:
\begin{align}
	\|\ml{N}[u]\|_{X_T}&\leqslant \varepsilon C_0\|(u_0,u_1)\|_{(H^{\sigma_{\min}}\cap L^1)\times (L^2\cap L^1)}+C_1(T)\|u\|_{X_T}^p,\label{Crucial-01}\\
	\|\ml{N}[u]-\ml{N}[\tilde{u}]\|_{X_T}&\leqslant C_1(T)\|u-\tilde{u}\|_{X_T}(\|u\|_{X_T}^{p-1}+\|\tilde{u}\|_{X_T}^{p-1}),\label{Crucial-02}
\end{align}
hold with a suitable function $C_1=C_1(T)>0$ depending on $T$. In the forthcoming subsection, we will demonstrate the crucial estimate
\begin{align}\label{Crucial}
C_1(T)\lesssim \int_0^{T}(1+\tau)^{-\frac{n(p-1)}{2\sigma_{\min}}}\,\mathrm{d}\tau+(1+T)^{-\frac{n(p-1)}{2\sigma_{\min}}+1}.
\end{align}
Let us now separate our discussion according to the value of $p$.
\begin{description}
	\item[The super-critical case $p>p_{\mathrm{crit}}$:] It is trivial to see that
	\begin{align*}
	\int_0^{T}(1+\tau)^{-\frac{n(p-1)}{2\sigma_{\min}}}\,\mathrm{d}\tau\lesssim 1\ \ \text{and}\ \ (1+T)^{-\frac{n(p-1)}{2\sigma_{\min}}+1}\lesssim 1
	\end{align*}
	uniformly in-time $T$, that is to say, $C_1(T)\leqslant C_2$ for any $T>0$, where $C_2$ is uniformly bounded with respect to $T$. Taking $T=+\infty$, the estimates \eqref{Crucial-01} and \eqref{Crucial-02} imply, respectively,
	\begin{align*}
	\|\ml{N}[u]\|_{X_{+\infty}}\leqslant\frac{3}{4}\kappa\ \ \text{and}\ \ \|\ml{N}[u]-\ml{N}[\tilde{u}]\|_{X_{+\infty}}\leqslant\frac{1}{2}\|u-\tilde{u}\|_{X_{+\infty}},
	\end{align*}
	if $4C_2\kappa^{p-1}\leqslant 1$. Namely, the choice of small size $\varepsilon$ leads to
	\begin{align*}
	0<\varepsilon\leqslant\varepsilon_1:=(4C_2)^{-\frac{1}{p-1}}\big(2C_0\|(u_0,u_1)\|_{(H^{\sigma_{\min}}\cap L^1)\times (L^2\cap L^1)}\big)^{-1}.
	\end{align*}
	Thus, it follows that $u\in\ml{B}_{\kappa}(X_{+\infty})$. As a byproduct, we also find  our desired estimate \eqref{Est-Decay}.
	\item[The sub-critical case $p<p_{\mathrm{crit}}$:] We may get
	\begin{align*}
	C_1(T)\leqslant C_3(1+T)^{-\frac{n(p-1)}{2\sigma_{\min}}+1},
	\end{align*}
	with a uniformly in-time suitable constant $C_3>0$. Analogously to the super-critical case, the operator $\ml{N}$ is a contraction on the ball $\ml{B}_{\kappa}(X_T)$ for $T$ satisfying
	\begin{align}\label{Est-T-sub}
	4C_3(1+T)^{-\frac{n(p-1)}{2\sigma_{\min}}+1}\kappa^{p-1}\leqslant 1 \ \ \Rightarrow \ \ 	T\lesssim \varepsilon^{-\frac{2\sigma_{\min}(p-1)}{2\sigma_{\min}-n(p-1)}}.
	\end{align}
	Therefore, for any $T>0$ such that \eqref{Est-T-sub} holds, a direct application of Banach fixed point argument leads to the existence and uniquely local in-time solution $u\in\ml{B}_{\kappa}(X_T)$. The lower bound estimates of lifespan $T_{\varepsilon}$ in the sub-critical case has been derived.
	\item[The critical case $p=p_{\mathrm{crit}}$:] For a uniformly in-time suitable constant $C_4>0$, it follows that
	\begin{align*}
	C_1(T)\leqslant C_4\ln(1+T).
	\end{align*}
	From the necessary condition (to ensure the local in-time existence of solution), one gets
	\begin{align*}
	4C_4\ln(1+T)\kappa^{p-1}\leqslant 1 \ \ \Rightarrow \ \ 	T\lesssim\exp(C\varepsilon^{-(p-1)}),
	\end{align*}
	which shows our desired lower bound estimate for lifespan $T_{\varepsilon}$ in the critical case.
\end{description}
All in all, to complete the proofs of Theorems \ref{Thm-GESDS} and \ref{Thm-Lower-Lifespan}, it remains to justify \eqref{Crucial} only.

\subsection{Proofs of Theorems \ref{Thm-GESDS} and \ref{Thm-Lower-Lifespan}}
\hspace{5mm}Let us recall the definition of $X_T$. Concerning $m\in\{1,2\}$, by using the fractional Gagliardo-Nirenberg inequality, we are able to conclude 
\begin{align*}
\|\,|u(\tau,\cdot)|^p\|_{L^m}=\|u(\tau,\cdot)\|_{L^{mp}}^p&\lesssim \|u(\tau,\cdot)\|_{L^2}^{(1-\beta_0)p}\|u(\tau,\cdot)\|_{\dot{H}^{\sigma_{\min}}}^{\beta_0p}\\
&\lesssim (1+\tau)^{\frac{n}{2m\sigma_{\min}}-\frac{np}{2\sigma_{\min}}}\|u\|_{X_T}^p,
\end{align*}
where $\beta_0:=(\frac{1}{2}-\frac{1}{mp})\frac{n}{\sigma_{\min}}\in[0,1]$ leading to $p\geqslant 2$, and $p\leqslant\frac{n}{n-2\sigma_{\min}}$ if $n>2\sigma_{\min}$. Let us consider $s\in\{0,\sigma_{\min}\}$. By employing the $(L^2\cap L^1)-L^2$ estimate in $[0,\frac{t}{2}]$ and the $L^2-L^2$ estimate in $[\frac{t}{2},t]$ from Proposition \ref{Prop-Decay-Estimates}, one deduces
\begin{align*}
\|u^{\nlin}(t,\cdot)\|_{\dot{H}^s}&\lesssim \int_0^{\frac{t}{2}}(1+t-\tau)^{-\frac{n+2s}{4\sigma_{\min}}}\|\,|u(\tau,\cdot)|^p\|_{L^2\cap L^1}\,\mathrm{d}\tau+\int_{\frac{t}{2}}^t(1+t-\tau)^{-\frac{s}{2\sigma_{\min}}}\|\,|u(\tau,\cdot)|^p\|_{L^2}\,\mathrm{d}\tau\\
&\lesssim (1+t)^{-\frac{n+2s}{4\sigma_{\min}}}\int_0^{\frac{t}{2}}(1+\tau)^{-\frac{n(p-1)}{2\sigma_{\min}}}\,\mathrm{d}\tau\,\|u\|_{X_T}^p\\
&\quad+(1+t)^{\frac{n}{4\sigma_{\min}}-\frac{np}{2\sigma_{\min}}}\int_{\frac{t}{2}}^t(1+t-\tau)^{-\frac{s}{2\sigma_{\min}}}\,\mathrm{d}\tau\,\|u\|_{X_T}^p\\
&\lesssim (1+t)^{-\frac{n+2s}{4\sigma_{\min}}}\Big(\int_0^{\frac{t}{2}}(1+\tau)^{-\frac{n(p-1)}{2\sigma_{\min}}}\,\mathrm{d}\tau+(1+t)^{-\frac{n(p-1)}{2\sigma_{\min}}+1}\Big)\|u\|_{X_T}^p,
\end{align*}
where the asymptotic relations $1+t-\tau\approx 1+t$ if $\tau\in[0,\frac{t}{2}]$ and $1+\tau\approx1+t$ if $\tau\in [\frac{t}{2},t]$. That is to say,
\begin{align*}
\|u^{\nlin}\|_{X_T}\lesssim\Big( \int_0^{T}(1+\tau)^{-\frac{n(p-1)}{2\sigma_{\min}}}\,\mathrm{d}\tau+(1+T)^{-\frac{n(p-1)}{2\sigma_{\min}}+1}\Big)\|u\|_{X_T}^p,
\end{align*}
 analogously,
\begin{align*}
\|\ml{N}[u]-\ml{N}[\tilde{u}]\|_{X_T}&\lesssim \Big( \int_0^{T}(1+\tau)^{-\frac{n(p-1)}{2\sigma_{\min}}}\,\mathrm{d}\tau+(1+T)^{-\frac{n(p-1)}{2\sigma_{\min}}+1}\Big)\|u-\tilde{u}\|_{X_T}\big(\|u\|_{X_T}^{p-1}+\|\tilde{u}\|_{X_T}^{p-1}\big).
\end{align*}
For this reason, our desired estimate \eqref{Crucial} is established.

\subsection{Proof of Theorem \ref{Thm-Large-Time}}
\hspace{5mm} Let us see that the global in-time solution demonstrated in Theorem \ref{Thm-GESDS} has the integral form
\begin{align}\label{Integral-Representation}
u(t,x)=\varepsilon\ml{K}_0(t,|D|)u_0(x)+\varepsilon\ml{K}_1(t,|D|)u_1(x)+\int_0^t\ml{K}_1(t-\tau,|D|)|u(\tau,x)|^p\,\mathrm{d}\tau.
\end{align}
Recalling the refined estimates for the linearized problem in Proposition \ref{Prop-Decay-Estimates}, in order to justify our desired error estimates, we have to consider each kernel in the nonlinear parts. In the following discussion, we consider large-time $t\gg1$ without more repetition.

First of all, we carry out a suitable decomposition (into five parts) as follows:
\begin{align*}
\int_0^t\ml{K}_1(t-\tau,|D|)[u(\tau,x)]^p\,\mathrm{d}\tau-G(t,x)\ml{P}_{|u|^p}=\sum\limits_{j\in\{1,2,\dots,5\}}A_j(t,x),
\end{align*}
where we took
\begin{align*}
A_1(t,x)&:=\int_0^{\frac{t}{2}}\big(\ml{K}_1(t-\tau,|D|)-\ml{G}(t-\tau,|D|)\big)|u(\tau,x)|^p\,\mathrm{d}\tau,\\
A_2(t,x)&:=\int_0^{\frac{t}{2}}\big(\ml{G}(t-\tau,|D|)-\ml{G}(t,|D|)\big)|u(\tau,x)|^p\,\mathrm{d}\tau,\\
A_3(t,x)&:=\int_0^{\frac{t}{2}}\Big(\ml{G}(t,|D|)|u(\tau,x)|^p-\int_{\mb{R}^n}|u(\tau,y)|^p\,\mathrm{d}y\,G(t,x)\Big)\mathrm{d}\tau,\\
A_4(t,x)&:=\int_{\frac{t}{2}}^t\ml{K}_1(t-\tau,|D|)|u(\tau,x)|^p\,\mathrm{d}\tau,\\
A_5(t,x)&:=-G(t,x)\int_{\frac{t}{2}}^{+\infty}\int_{\mb{R}^n}|u(\tau,y)|^p\,\mathrm{d}y\,\mathrm{d}\tau.
\end{align*}
Let us denote
\begin{align*}
\alpha_{\min}:=\begin{cases}
\min\{2-2\sigma,2\sigma\}&\text{if}\ \ \sigma\in(0,1),\\
\min\{2,2\sigma-2\}&\text{if}\ \ \sigma\in(1,+\infty),
\end{cases}
\end{align*}
so the proof of \eqref{Eq-01} has shown
\begin{align*}
\|\ml{K}_1(t,|D|)g(\cdot)-\ml{G}(t,|D|)g(\cdot)\|_{\dot{H}^s}\lesssim t^{-\frac{n+2s+2\alpha_{\min}}{4\sigma_{\min}}}\|g\|_{L^2\cap L^1}.
\end{align*}
\begin{itemize}
	\item[$\bullet$] According to Proposition \ref{Prop-Decay-Estimates} and the decay estimate \eqref{Est-Decay} of global in-time solution $u$, the first term can be estimated by
	\begin{align*}
	t^{\frac{n+2s}{4\sigma_{\min}}}\|A_1(t,\cdot)\|_{\dot{H}^s}&\lesssim t^{\frac{n+2s}{4\sigma_{\min}}} \int_0^{\frac{t}{2}}(t-\tau)^{-\frac{n+2s+2\alpha_{\min}}{4\sigma_{\min}}}\|\,|u(\tau,\cdot)|^p\|_{L^2\cap L^1}\,\mathrm{d}\tau\\
	&\lesssim t^{-\frac{\alpha_{\min}}{2\sigma_{\min}}}\int_0^{\frac{t}{2}}(1+\tau)^{-\frac{n(p-1)}{2\sigma_{\min}}}\,\mathrm{d}\tau\,\|u\|_{X_{+\infty}}^p\\
	&\lesssim\varepsilon^p\,t^{-\frac{\alpha_{\min}}{2\sigma_{\min}}}\|(u_0,u_1)\|_{(H^{\sigma_{\min}}\cap L^1)\times (L^2\cap L^1)}^p.
	\end{align*}
	\item By using the mean value theorem with respect to $t$, i.e.
	\begin{align*}
	\ml{G}(t-\tau,|D|)-\ml{G}(t,|D|)=-\tau\ml{G}_t(t-\theta_1\tau,|D|)\ \ \text{with}\ \ \theta_1\in(0,1),
	\end{align*}
	the second term can be estimated by
	\begin{align*}
	t^{\frac{n+2s}{4\sigma_{\min}}}\|A_2(t,\cdot)\|_{\dot{H}^s}&\lesssim t^{\frac{n+2s}{4\sigma_{\min}}}\int_0^{\frac{t}{2}}\tau (t-\tau)^{-\frac{n+2s}{4\sigma_{\min}}-1}\|\,|u(\tau,\cdot)|^p\|_{L^2\cap L^1}\,\mathrm{d}\tau\\
	&\lesssim\varepsilon^p\, t^{-1}\int_0^{\frac{t}{2}}(1+\tau)^{1-\frac{n(p-1)}{2\sigma_{\min}}}\,\mathrm{d}\tau\,\|u\|_{X_{+\infty}}^p\\
	&\lesssim \varepsilon^p\|(u_0,u_1)\|_{(H^{\sigma_{\min}}\cap L^1)\times (L^2\cap L^1)}^p\times\begin{cases}
	t^{1-\frac{n(p-1)}{2\sigma_{\min}}}&\text{if}\ \ p<p_{\mathrm{crit}}(n,2\sigma_{\min}),\\
	t^{-1}\ln t&\text{if}\ \ p=p_{\mathrm{crit}}(n,2\sigma_{\min}),\\
	t^{-1}&\text{if}\ \ p>p_{\mathrm{crit}}(n,2\sigma_{\min}),
	\end{cases}
	\end{align*}
	whose right-hand sides tend to zero as $t\to+\infty$.
	\item By the same way as \eqref{Eq-02}, i.e. an application of mean value theorem with respect to $x$ and the additional separation via $h(t):=t^{\frac{1}{4\sigma_{\min}}}$, one derives
	\begin{align*}
	t^{\frac{n+2s}{4\sigma_{\min}}}\|A_3(t,\cdot)\|_{\dot{H}^s}&\lesssim t^{-\frac{1}{4\sigma_{\min}}}\int_0^{\frac{t}{2}}\|\,|u(\tau,\cdot)|^p\|_{L^1}\,\mathrm{d}\tau+\int_0^{\frac{t}{2}}\|\,|u(\tau,\cdot)|^p\|_{L^1(|x|\geqslant h(\tau))}\,\mathrm{d}\tau\\
	&\lesssim \varepsilon^p\, t^{-\frac{1}{4\sigma_{\min}}}\int_0^{+\infty}(1+\tau)^{-\frac{n(p-1)}{2\sigma_{\min}}}\,\mathrm{d}\tau\,\|(u_0,u_1)\|_{(H^{\sigma_{\min}}\cap L^1)\times (L^2\cap L^1)}^p\\
	&\quad+\int_0^{+\infty}\|\,|u(\tau,\cdot)|^p\|_{L^1(|x|\geqslant h(\tau))}\,\mathrm{d}\tau.
	\end{align*}
	According to the fact that
	\begin{align*}
	\|\,|u(\tau,\cdot)|^p\|_{L^1([0,+\infty)\times\mb{R}^n)}\lesssim \varepsilon\,\int_0^{+\infty}(1+\tau)^{-\frac{n(p-1)}{2\sigma_{\min}}}\,\mathrm{d}\tau\,\|(u_0,u_1)\|_{(H^{\sigma_{\min}}\cap L^1)\times (L^2\cap L^1)}^p<+\infty 
	\end{align*}
	from $p>p_{\mathrm{crit}}$, we claim
	\begin{align*}
	\lim\limits_{t\to+\infty}\int_0^{+\infty}\|\,|u(\tau,\cdot)|^p\|_{L^1(|x|\geqslant h(\tau))}\,\mathrm{d}\tau=0.
	\end{align*}
	It tells us that
	\begin{align*}
	\lim\limits_{t\to+\infty}t^{\frac{n+2s}{4\sigma_{\min}}}\|A_3(t,\cdot)\|_{\dot{H}^s}=0.
	\end{align*}
	\item After applying the $L^2-L^2$ estimate in Proposition \ref{Prop-Decay-Estimates}, the fourth term can be estimated by
	\begin{align*}
	t^{\frac{n+2s}{4\sigma_{\min}}}\|A_4(t,\cdot)\|_{\dot{H}^s}&\lesssim t^{\frac{n+2s}{4\sigma_{\min}}}\int_{\frac{t}{2}}^t(1+t-\tau)^{-\frac{s}{2\sigma_{\min}}}\|\,|u(\tau,\cdot)|^p\|_{L^2}\,\mathrm{d}\tau\\
	&\lesssim \varepsilon^p\, t^{-\frac{n(p-1)}{2\sigma_{\min}}+1}\|(u_0,u_1)\|_{(H^{\sigma_{\min}}\cap L^1)\times (L^2\cap L^1)}^p,
	\end{align*}
	whose right-hand side tends to zero as $t\to+\infty$ thanks to $p>p_{\mathrm{crit}}$.
	\item The application of the decay estimates \eqref{Eq-04} and \eqref{Est-Decay} shows
	\begin{align*}
	t^{\frac{n+2s}{4\sigma_{\min}}}\|A_5(t,\cdot)\|_{\dot{H}^s}	&\lesssim \varepsilon^p\,t^{\frac{n+2s}{4\sigma_{\min}}}\int_{\frac{t}{2}}^{+\infty}(1+\tau)^{-\frac{n(p-1)}{2\sigma_{\min}}}\,\mathrm{d}\tau\,\|G(t,\cdot)\|_{\dot{H}^s}\|(u_0,u_1)\|_{(H^{\sigma_{\min}}\cap L^1)\times (L^2\cap L^1)}^p\\
	&\lesssim\varepsilon^p\,t^{-\frac{n(p-1)}{2\sigma_{\min}}+1}\|(u_0,u_1)\|_{(H^{\sigma_{\min}}\cap L^1)\times (L^2\cap L^1)}^p.
	\end{align*}
	Again, its right-hand side tends to zero as $t\to+\infty$.
\end{itemize}
All in all, we have just proved that
\begin{align*}
\lim\limits_{t\to+\infty}t^{\frac{n+2s}{4\sigma_{\min}}}\big\|u^{\nlin}(t,\cdot)-G(t,\cdot)\ml{P}_{|u|^p}\big\|_{\dot{H}^s}=0.
\end{align*}

Combining the derived estimate for $u^{\lin}(t,\cdot)$ in Proposition \ref{Prop-Decay-Estimates}, via the integral representation \eqref{Integral-Representation}, we complete the derivation of asymptotic profiles for the global in-time solution. Additionally, using the triangle inequality, one arrives at
\begin{align*}
\|u(t,\cdot)\|_{\dot{H}^s}&\geqslant \|G(t,\cdot)\|_{\dot{H}^s}\big|\varepsilon P_{u_{0+1}}+\ml{P}_{|u|^p}\big|-\big\|u(t,\cdot)-G(t,\cdot)\big( \varepsilon P_{u_{0+1}}+\ml{P}_{|u|^p}\big)\big\|_{\dot{H}^s}\\
&\gtrsim t^{-\frac{n+2s}{4\sigma_{\min}}}\big|\varepsilon P_{u_{0+1}}+\ml{P}_{|u|^p}\big|-o\big(t^{-\frac{n+2s}{4\sigma_{\min}}}\big)\gtrsim t^{-\frac{n+2s}{4\sigma_{\min}}}\big|\varepsilon P_{u_{0+1}}+\ml{P}_{|u|^p}\big|
\end{align*}
as large-time $t\gg1$ via \eqref{Eq-04}. Hence, the sharp large-time lower bound estimate for the global in-time solution $u$ can be demonstrated.

\section{Blow-up for semilinear damped wave equations driven by $\ml{L}_{a,b}$}\setcounter{equation}{0}\label{Section-Blow-up}
\hspace{5mm} At first, let us introduce the time-dependent test functions $\eta= \eta(t) \in \mathcal{C}_0^\infty([0,+\infty))$ having the property $\eta(t)=1$ if $0\leqslant t\leqslant \frac{1}{2}$; $\eta(t)=0$ if $t\geqslant1$; and decreasing if $\frac{1}{2}\leqslant t\leqslant 1$; together with the condition
\begin{align}\label{Condition.Eta}
	\eta^{-\frac{p'}{p}}(t)\big(|\eta'(t)|^{p'}+|\eta''(t)|^{p'}\big)\leqslant C\ \ \text{for any}\ \ t\in[\tfrac{1}{2},1].
\end{align}
Also, we define the radial space-dependent test function $\varphi=\varphi(x):=\langle x\rangle^{-n-2\sigma_0}$ with
\begin{align*}
\sigma_0:= \begin{cases}
	\text{a constant in }(0,1) &\text{if}\ \  \sigma\in\mb{N}_+, \\
	\sigma-[\sigma] &\text{if}\ \  \sigma\in\mb{R}_+\backslash\mb{N}_+.
\end{cases}
\end{align*}
It is clear to see the estimate
\begin{align}\label{Condition.Laplace}
	|(-\Delta)^{\sigma} \langle x\rangle^{-n-2\sigma_0}| \lesssim \langle x\rangle^{-n-2\sigma_0}.
\end{align}
Let $R$ and $K$ be large parameters in $[0,+\infty)$. We introduce the test function  $\phi_R(t,x):= \eta_R(t)\, \varphi_R(x)$, where $\eta_R(t):= \eta(R^{-2\sigma_{\min}}t)$ and $\varphi_R(x):= \varphi(K^{-1}R^{-1}x)$, then define the functionals
\begin{align*}
	J_R := \int_0^{+\infty}\int_{\mb{R}^n}|u(t,x)|^p \phi_R(t,x)\,\mathrm{d}x\,\mathrm{d}t\ \ \text{and}\ \ 
	\tilde{J}_R := \int_{R^{2\sigma_{\min}}/2}^{R^{2\sigma_{\min}}}\int_{\mb{R}^n}|u(t,x)|^p \phi_R(t,x)\,\mathrm{d}x\,\mathrm{d}t.
\end{align*}
Assume that $u$ is a global in-time Sobolev solution from $\mathcal{C}([0,+\infty),L^2)$ to \eqref{Eq-Nonlinear-Local-Nonlocal}. After multiplying both sides of \eqref{Eq-Nonlinear-Local-Nonlocal} by $\varphi_R$, we carry out integration by parts to achieve
\begin{align}
	0\leqslant J_R &= -\int_{\mb{R}^n} u_{0+1}(x)\,\varphi_R(x)\,\mathrm{d}x + \int_{R^{2\sigma_{\min}}/2}^{R^{2\sigma_{\min}}}\int_{\mb{R}^n}u(t,x)\, \eta''_R(t)\, \varphi_R(x)\,\mathrm{d}x\,\mathrm{d}t \notag \\
	&\quad -a\int_0^{+\infty}\int_{\mb{R}^n} u(t,x)\,\eta_R(t)\, \Delta\varphi_R(x)\,\mathrm{d}x\,\mathrm{d}t+ b\int_0^{+\infty}\int_{\mb{R}^n} u(t,x)\,\eta_R(t) \,(-\Delta)^{\sigma} \varphi_R(x)\,\mathrm{d}x\,\mathrm{d}t \notag \\
	&\quad -\int_{R^{2\sigma_{\min}}/2}^{R^{2\sigma_{\min}}}\int_{\mb{R}^n}\, u(t,x) \,\eta'_R(t)\, \varphi_R(x)\,\mathrm{d}x\,\mathrm{d}t \notag \\
	&=: -\int_{\mb{R}^n} u_1(x)\,\varphi_R(x)\,\mathrm{d}x+ J_{1,R}- J_{2,R}+ J_{3,R}- J_{4,R}, \label{the2.3-1}
\end{align}
where we notice that the  identity
\begin{align*}
\int_{\mb{R}^n} \varphi_R(x)\, (-\Delta)^{\sigma} u(t,x)\,\mathrm{d}x= \int_{\mb{R}^n} u(t,x)\, (-\Delta)^{\sigma} \varphi_R(x)\,\mathrm{d}x
\end{align*}
holds since $\varphi_R \in H^{2\sigma}$ and $u \in \mathcal{C}([0,+\infty),L^2)$. Applying the H\"{o}lder inequality with $\frac{1}{p}+\frac{1}{p'}=1$, we derive
\begin{align*}
	|J_{1,R}| &\leqslant \int_{R^{2\sigma_{\min}}/2}^{R^{2\sigma_{\min}}}\int_{\mb{R}^n} |u(t,x)|\, |\eta''_R(t)|\, \varphi_R(x) \, \mathrm{d}x\,\mathrm{d}t \\
	&\lesssim \Big(\int_{R^{2\sigma_{\min}}/2}^{R^{2\sigma_{\min}}}\int_{\mb{R}^n} |u(t,x)\,\phi^{\frac{1}{p}}_R(t,x)|^p\, \mathrm{d}x\,\mathrm{d}t\Big)^{\frac{1}{p}} \Big(\int_{R^{2\sigma_{\min}}/2}^{R^{2\sigma_{\min}}}\int_{\mb{R}^n} |\phi^{-\frac{1}{p}}_R(t,x)\, \eta''_R(t)\, \varphi_R(x)|^{p'}\, \mathrm{d}x\,\mathrm{d}t\Big)^{\frac{1}{p'}} \\
	&\lesssim \tilde{J}_{R}^{\frac{1}{p}} \Big(\int_{R^{2\sigma_{\min}}/2}^{R^{2\sigma_{\min}}}\int_{\mb{R}^n} \eta_R^{-\frac{p'}{p}}(t)\, |\eta''_R(t)|^{p'} \varphi_R(x)\, \mathrm{d}x\,\mathrm{d}t\Big)^{\frac{1}{q'}}.
\end{align*}
Using the change of variables $\tilde{t}:= R^{-2\sigma_{\min}}t$ and $\tilde{x}:= K^{-1}R^{-1}x$ leads to 
\begin{equation} \label{the2.3-2}
	|J_{1,R}| \lesssim \tilde{J}_{R}^{\frac{1}{p}} R^{-4\sigma_{\min}+ \frac{n+2\sigma_{\min}}{p'}}K^{\frac{n}{p'}}\Big(\int_{\mb{R}^n} \langle \tilde{x} \rangle^{-n-2\sigma_0}\, \mathrm{d}\tilde{x}\Big)^{\frac{1}{p'}},
\end{equation}
due to the relation $ \eta''_R(t)= R^{-4\sigma_{\min}}\eta''(\tilde{t})$ and the assumption \eqref{Condition.Eta}. In similar arguments, we also arrive at the following estimates:
\begin{align}
	|J_{2,R}| &\lesssim J_R^{\frac{1}{p}} R^{-2+ \frac{n+2\sigma_{\min}}{p'}}K^{\frac{n}{p'}-2}, \label{the2.3-3} \\
	|J_{3,R}| &\lesssim J_R^{\frac{1}{p}} R^{-2\sigma+ \frac{n+2\sigma_{\min}}{p'}}K^{\frac{n}{p'}-2\sigma}, \label{the2.3-4} \\
	|J_{4,R}| &\lesssim \tilde{J}_R^{\frac{1}{p}} R^{-2\sigma_{\min}+ \frac{n+2\sigma_{\min}}{p'}}K^{\frac{n}{p'}}, \label{the2.3-5}
\end{align}
thanks to the relations $\eta'_R(t)= R^{-2\sigma}\eta'(\tilde{t})$, $\Delta \varphi_R(x)= K^{-2}R^{-2}\Delta \varphi(\tilde{x})$ and
\begin{align*}
(-\Delta)^{\sigma}\varphi_R(x)= K^{-2\sigma}R^{-2\sigma}(-\Delta)^{\sigma}\varphi(\tilde{x})
\end{align*}
combined with the assumptions \eqref{Condition.Eta} and \eqref{Condition.Laplace}. Since the assumption on the initial data, there exists a constant $R_0>0$ such that
\begin{align} \label{the2.3-6}
	\int_{\mb{R}^n} u_{0+1}(x)\, \varphi_R(x)\,\mathrm{d}x >0\ \ \mbox{for any}\ \ R>R_0.
\end{align}
For this reason, linking all gained estimates from \eqref{the2.3-1} to \eqref{the2.3-6}, we have shown that
\begin{align} \label{the2.3-7}
	J_R+ \int_{\mb{R}^n} u_{0+1}(x)\, \varphi_R(x)\,\mathrm{d}x \lesssim R^{-2\sigma_{\min}+ \frac{n+2\sigma_{\min}}{p'}}\big(\tilde{J}_R^{\frac{1}{p}}K^{\frac{n}{p'}}+ J_R^{\frac{1}{p}}K^{\frac{n}{p'}-2\sigma_{\min}}\big)
\end{align}
for any $R>R_0$, which implies
\begin{align}\label{the2.3-8}
J_R \lesssim R^{-2\sigma_{\min}+ \frac{n+2\sigma_{\min}}{p'}} J_R^{\frac{1}{p}}K^{\frac{n}{p'}}\ \ \Rightarrow\ \ 	J_R^{\frac{1}{p'}} \lesssim R^{-2\sigma_{\min}+ \frac{n+2\sigma_{\min}}{p'}}K^{\frac{n}{p'}}
\end{align}
due to $\tilde{J}_R \leqslant J_R$.
Let us divide our consideration into two cases.
\begin{description}
	\item[The sub-critical case $p< p_{\mathrm{crit}}$:] This condition leads to the relation $-2\sigma_{\min}+ \frac{n+2\sigma_{\min}}{p'} < 0$. Hence, choosing $K=1$ and passing $R \to +\infty$ in \eqref{the2.3-8}, one gets
	\begin{align*}
		\int_0^{+\infty}\int_{\mb{R}^n}|u(t,x)|^p \,\mathrm{d}x\,\mathrm{d}t= 0,
	\end{align*}
	i.e. $u= 0$, which is a contradiction to the assumption of the initial data.
	\item[The critical case $p= p_{\mathrm{crit}}$:]  It follows that $-2\sigma_{\min}+ \frac{n+2\sigma_{\min}}{p'} = 0$.  
	At the first choice, we take $K=1$ into \eqref{the2.3-8} to see that $J_R\leqslant 1$, i.e. $J_R$ is uniformly bounded. As as result, one obtains $\displaystyle\lim_{R\to+\infty} \tilde{J}_R=0$. Next, for a sufficiently large number $K>0$, we recall \eqref{the2.3-7} and apply Young's inequality to get
	%	\begin{align*}
		%	J_R+ \int_{\mb{R}^n} u_1(x)\, \varphi_R(x)\,\mathrm{d}x \leqslant C_1\tilde{J}_R^{\frac{1}{p}}K^{\frac{n}{p'}}+ \frac{J_R}{p}+ C_2K^{\frac{n}{p'}-2\sigma_{\min}},
		%	\end{align*}
	%	which is equivalent to
	\begin{align*}
		J_R+ p'\int_{\mb{R}^n} u_{0+1}(x)\, \varphi_R(x)\,\mathrm{d}x \leqslant\widetilde{C}_1p'\tilde{J}_R^{\frac{1}{p}}K^{\frac{n}{p'}}+ \widetilde{C}_2p'K^{\frac{n}{p'}-2\sigma_{\min}}.
	\end{align*}
	Thus, it entails immediately the following estimate:
	\begin{align*}
		0< \int_{\mb{R}^n} u_{0+1}(x)\, \varphi_R(x)\,\mathrm{d}x \leqslant \widetilde{C}_1\tilde{J}_R^{\frac{1}{p}}K^{\frac{n}{p'}}+ \widetilde{C}_2K^{\frac{n}{p'}-2\sigma_{\min}}.
	\end{align*}
	Finally, letting $R\to +\infty$ first and then $K\to +\infty$ in the last inequality we conclude that $\int_{\mb{R}^n} u_{0+1}(x) \,\mathrm{d}x= 0$, 
	due to the fact $\frac{n}{p'}-2\sigma_{\min}<0$, which is a contradiction to the assumption of the initial data again.
\end{description}
In conclusion, we have established the blow-up results in Theorem \ref{Thm-Blow-up} in both the sub-critical case and the critical case. \medskip

To verify the lifespan estimates, we suppose that $u$ is a local in-time Sobolev solution to \eqref{Eq-Nonlinear-Local-Nonlocal} in $u \in \mathcal{C}([0,T),L^2)$.
\begin{description}
	\item[The sub-critical case $p< p_{\mathrm{crit}}$:] We repeat the same way as above to obtain
	\begin{align}
		J_R+ c\varepsilon \leqslant C J_R^{\frac{1}{p}} R^{-2\sigma_{\min}+ \frac{n+2\sigma_{\min}}{p'}}\ \ \Rightarrow \ \ c\varepsilon \leqslant C J_R^{\frac{1}{p}} R^{-2\sigma_{\min}+ \frac{n+2\sigma_{\min}}{p'}}- J_R,
		\label{the2.4-1}
	\end{align}
	where a suitable constant $c$ is chosen such that
	\begin{align*}
		\int_{\mb{R}^n} u_{0+1}(x)\, \varphi_R(x)\,\mathrm{d}x> c> 0 \ \ \mbox{for any}\ \ R> R_0.
	\end{align*}
	After applying the elementary inequality $By^\gamma- y \leqslant B^{\frac{1}{1-\gamma}}$ for any $B>0$, $y\geqslant 0$, and $0<\gamma<1$ to
	\eqref{the2.4-1}, we conclude that
	\begin{align*}
		\varepsilon \leqslant CR^{-2\sigma_{\min} p'+ n+ 2\sigma_{\min}}= CT^{-\frac{2\sigma_{\min}- n(p-1)}{2\sigma_{\min}(p-1)}}
	\end{align*}
	with $R= T^{\frac{1}{2\sigma_{\min}}}$. Letting $T \to T_\varepsilon$, it gives the desired upper bound estimate for lifespan.
	\item[The critical case $p= p_{\mathrm{crit}}$:] To verify the lifespan upper estimate in the critical case, we just need to repeat the same procedure as that in \cite[Theorem 2.4 or Proposition 1.1]{Chen-Girardi=2025}. The smaller exponent in $\ml{L}_{a,b}$ plays a dominant role because it produces a smaller power of $R^{-1}$. For the sake of readability, we omits its details.
\end{description}

\section*{Acknowledgments}
Wenhui Chen is supported in part by the National Natural Science Foundation of China (grant No. 12301270), Guangdong Basic and Applied Basic Research Foundation (grant No. 2025A1515010240). This research of Tuan Anh Dao is funded by Vietnam National Foundation for Science and Technology Development (NAFOSTED).

\end{document}